\documentclass{amsart}

\usepackage{amssymb,color}
\usepackage{amsfonts}
\usepackage{amsmath}
\usepackage{euscript}
\usepackage{enumerate}
\usepackage{pdfsync}
\newtheorem{theorem}{Theorem}[section]
\newtheorem{lemma}[theorem]{Lemma}

\newtheorem{prop}[theorem]{Proposition}
\newtheorem{cor}[theorem]{Corollary}
\newtheorem{exa}[theorem]{Example}

\newtheorem*{Theorem1'}{Theorem 1'}

\theoremstyle{definition}

\theoremstyle{remark}

\def\Z{{\bf Z}}

\def\GL{{\rm GL}}

\def\Inn{{\rm Inn}}
\def\Aut{{\rm Aut}}
\def\End{{\rm End}}
\def\deg{{\rm deg\,}}

\def\Hol{{\rm Hol}}
\def\a{\alpha}
\def\b{\beta}

\begin{document}

\title[Isomorphism
of relative holomorphs and matrix similarity]{Isomorphism
of relative holomorphs and matrix similarity}

\author{Volker Gebhardt}
\address{Centre for Research in Mathematics and
Data Science, Western Sydney University, Australia}
\email{v.gebhardt@westernsydney.edu.au}

\author{Alberto J. Hernandez Alvarado}
\address{CIMPA, Universidad de Costa Rica, Costa Rica}
\email{albertojose.hernandez@ucr.ac.cr}

\author{Fernando Szechtman}
\address{Department of Mathematics and Statistics, University of Regina, Canada}
\email{fernando.szechtman@gmail.com}
\thanks{The third author was supported in part by an NSERC discovery grant}





\begin{abstract} Let $V$ be a finite dimensional vector space over the field with $p$ elements, where $p$ is a prime number.
Given arbitrary $\a,\b\in\GL(V)$, we consider the semidirect products $V\rtimes\langle \a\rangle$ and $V\rtimes\langle \b\rangle$, and show
that if $V\rtimes\langle \a\rangle$ and $V\rtimes\langle \b\rangle$ are isomorphic, 
then $\a$ must be similar to a power of $\b$ that generates
the same subgroup as $\b$; that is, if $H$ and $K$ are cyclic subgroups of 
$\GL(V)$ such that $V\rtimes H\cong V\rtimes K$, then
$H$ and $K$ must be conjugate subgroups of $\GL(V)$.
If we remove the cyclic condition, there exist examples of {\em non-isomorphic},
let alone non-conjugate, subgroups $H$ and $K$ of $\GL(V)$ such that $V\rtimes H\cong V\rtimes K$.
Even if we require that non-cyclic subgroups $H$ and~$K$ of $\GL(V)$ be abelian, we may still have 
$V\rtimes H\cong V\rtimes K$ with $H$ and~$K$ non-conjugate in $\GL(V)$, but in this case, $H$ and $K$ must
at least be isomorphic. If we replace $V$ by a free module $U$ over $\Z/p^m\Z$ of finite rank, with $m>1$, 
it may happen that $U\rtimes H\cong U\rtimes K$ for non-conjugate cyclic subgroups of $\GL(U)$.
If we completely abandon our requirements on~$V$, a sufficient criterion is given for a finite group~$G$
to admit non-conjugate cyclic  subgroups $H$ and~$K$ of $\mathrm{Aut}(G)$ such that 
$G\rtimes H\cong G\rtimes K$. This criterion is satisfied by many groups.
\end{abstract}

\maketitle

\section{Introduction}

We fix throughout a prime number $p$, and write $F$ for the field with $p$ elements as well as
$V$ for an $F$-vector space $V$ of finite dimension $n>0$.
Given an automorphism $\a$ of $V$, we may consider the semidirect product $G_\a=V\rtimes\langle \a\rangle$, where
$$
\a v \a^{-1}=\a(v),\quad v\in V.
$$
Likewise, given $\b\in\GL(V)$, we have the semidirect product $G_\b=V\rtimes\langle \b\rangle$. 
It is well known that if $\langle \a\rangle$ and $\langle \b\rangle$ 
are conjugate subgroups of $\GL(V)$, then $G_\a$ is isomorphic
to $G_\b$. In Theorem~\ref{lindo} we prove the converse: if $G_\a\cong G_\b$, then $\langle \a\rangle$ and $\langle \b\rangle$ 
must be conjugate in $\GL(V)$. Here $\langle \a\rangle$ and~$\langle \b\rangle$ are conjugate if and only if
$\a$ is similar to $\b^i$  for  some integer $i$ coprime to the order of~$\b$, i.e., such that
$\langle \b^i\rangle=\langle \b\rangle$. The proof of Theorem~\ref{lindo} is somewhat subtle. 
A more transparent argument is given in Theorem \ref{lindo2}, provided 
$\a$ and~$\b$ are unipotent, in which case $\a$ and $\b$ must themselves be similar. 
Theorem \ref{lindo}
seems to be exceptional in the realm of group theory, 
in the sense that any changes to the given linear algebra setting will tend to make it fail, as explained below. 

Given an arbitrary group $G$, we may consider its holomorph $\Hol(G)=G\rtimes\mathrm{Aut}(G)$, where
$$
\a g \a^{-1}=\a(g),\quad \a\in \mathrm{Aut}(G), g\in G.
$$
For a subgroup $H$ of $\mathrm{Aut}(G)$, we have the relative holomorph $\Hol(G,H)=G\rtimes H$,
viewed as a subgroup of $\Hol(G)$. If $K$ is also a subgroup of $\mathrm{Aut}(G)$, it is 
well known that if $H$ and~$K$ are conjugate in $\mathrm{Aut}(G)$, then $\Hol(G,H)$ and $\Hol(G,K)$
are conjugate in $\Hol(G)$, and hence isomorphic. The converse is false in general. In Example \ref{e3} we exhibit
{\em non-isomorphic} non-cyclic subgroups $H$ and $K$ of $\GL(V)$ such that
$\Hol(V,H)\cong\Hol(V,K)$, provided $p$ is odd and $n\geq 2$. The case $p=2$ and $n\geq 4$ is dealt with
in Example \ref{e7}. Moreover, Example~\ref{e9} gives non-conjugate, non-cyclic, abelian subgroups $H$ and $K$ of $\GL(V)$ such that
$\Hol(V,H)\cong\Hol(V,K)$ for any $n\geq 6$. On the other hand, in Corollary \ref{suma} we show that if $H$ and $K$ are abelian subgroups of 
$\GL(V)$ such that $\Hol(V,H)\cong\Hol(V,K)$
and the sum of all subspaces $(h-1)V$, with $h\in H$, is equal to $V$, then $H$ is necessarily conjugate to $K$. Furthermore, 
according to Proposition \ref{abe}, $\Hol(V,H)\cong\Hol(V,K)$ always
forces $H\cong K$ when $H$ and $K$ abelian. 
Example \ref{final} shows that Theorem~\ref{lindo} fails, in general, if $V$ is replaced by a free module $U$ over $\Z/p^m\Z$
of finite rank $n$. Indeed, when $p=2, m=3, n=4$, we exhibit an explicit automorphism $\a$ of $U$ such that 
$\Hol(U,\a)\cong\Hol(U,\b)$ for exactly two conjugacy classes of cyclic subgroups $\langle \b\rangle$ of $\GL(U)$.
When used in conjunction with Lemma \ref{nicecase0}, Examples \ref{final} and \ref{e9} 
reveal instances when $\Hol(G,H)\cong\Hol(G,K)$ but no isomorphism between them sends $G$ back into itself.

If we completely abandon our requirements on $V$, Example \ref{e1} gives a sufficient criterion for a group $G$
to admit cyclic subgroups $H$ and $K$ such that $\Hol(G,H)\cong\Hol(G,K)$ but $H$ and $K$
are not conjugate in $\mathrm{Aut}(G)$. This is illustrated  with various instances of such $G$, $H$, and $K$. 
Finally, Example \ref{e6} lists a few groups $G$ such that the existence of an isomorphism between $\Hol(G,H)$ and $\Hol(G,K)$, for 
arbitrary subgroups $H$ and $K$ of $\mathrm{Aut}(G)$, forces $H$ to be similar to $K$.

The holomorph of a group appeared early in the literature to produce examples of complete groups,
that is, groups having trivial center and only inner automorphisms. In 1908 Miller \cite{M} showed that 
the holomorph of any finite abelian group of odd order is complete, based on prior work of Burnside \cite{B}
on the subject. Miller's result was
extended to show that the holomorph, or a relative holomorph, of other classes of abelian groups, not necessarily finite, are also
complete, see \cite{Be, C, H, P}, for instance. A long-standing problem in this regard was the existence of complete
groups of odd order, settled positively by Dark \cite{D} in 1975.

We are concerned here with a specific case of the general problem of finding necessary and sufficient conditions
for two semidirect products $A\rtimes_f B$ and $C\rtimes_g D$ to be isomorphic, where $f:B\to\mathrm{Aut}(A)$
and $g:D\to\mathrm{Aut}(C)$ are homomorphisms. No general answer is known.
When $A=C$ and $B=D$ this problem was investigated by Taunt \cite{T} and Kuzennyi \cite{K}.
If $B=\mathrm{Aut}(A)$, $D=\mathrm{Aut}(C)$, and $f$ and $g$ are the identity maps, Mills showed in \cite{Mi} (resp. \cite{Mi2})
that $\Hol(A)\cong\Hol(C)$ forces $A\cong C$ when both $A$ and $C$ are finitely generated abelian
groups (resp. when $A$ or $C$ are finite abelian groups). Mills \cite{Mi} also pointed out that 
when $n\geq 3$, the non-isomorphic dihedral and generalized quaternion groups of order $4n$ have isomorphic holomorphs.
A proof can be found in Kohl's paper \cite[Proposition 3.10]{Ko}. In Miller's work \cite{M}, the holomorph
of a group $G$ is viewed as the normalizer in the symmetric group $S(G)$ of the left (resp. right)
regular representation of $G$. Miller referred to the normalizer in $S(G)$ of $\Hol(G)$ as the multiple holomorph, say $M(G)$,
of $G$. The quotient group $T(G)=M(G)/\Hol(G)$ has interesting properties and has 
received considerable attention recently. The 
structure of $T(G)$ was determined by Kohl \cite{Ko2} for dihedral and generalized quaternion groups, and by
Caranti and Dalla Volta \cite{CD} for finite perfect groups with trivial center. 

The conjugacy of subgroups $H$ and $K$ of $\mathrm{Aut}(G)$ is an obvious sufficient condition for $\Hol(G,H)$ to be isomorphic 
to $\Hol(G,K)$,
and in the present paper we show that it is also necessary, provided $H$ and $K$ are cyclic and $G=V$.
Our examples show that one cannot deviate much from the stated hypotheses for 
the necessity of conjugacy to hold. One may use our results in the classification, up to isomorphism, of the relative holomorphs
of an elementary abelian group such as $V$. This is a difficult problem, as attested by one of its simplest cases \cite{VO},
namely when $p=2$ and $n=4$, in which case there are 138 relative holomorphs.


\section{Background from linear algebra}\label{la}

Given $\alpha\in\End(V)$, we write $m_\alpha\in F[X]$ for the minimal polynomial of $\alpha$.
Given a group $G$ and $g\in G$, we let $o(g)$ stand for the order of $g$.

\begin{lemma}\label{orden} Let $\alpha\in\GL(V)$. Then $p\nmid o(\alpha)$ if and only if $m_\alpha$ is square-free.
\end{lemma}

\begin{proof} Let $m=o(\alpha)$. Then $m$ is the smallest natural number such that
$$
m_\alpha\mid (X^m-1).
$$
Now $X^m-1$ and its formal derivative, namely $mX^{m-1}$, are relatively prime if and only if $p\nmid m$.
\end{proof}

\begin{lemma}\label{psimilar} Let $\alpha\in\GL(V)$. Then $\alpha$ is similar to $\alpha^p$ if and only if $p\nmid o(\alpha)$.
\end{lemma}

\begin{proof} If $p\mid o(\alpha)$ then $\alpha$ and $\alpha^p$ have different orders, whence they are not similar.

Suppose $p\nmid o(\alpha)$. Then by Lemma \ref{orden} the invariant factors of $\alpha$ are all square-free.
Thus the rational canonical form of $\alpha$ is the direct sum of companion matrices $C_f$ to square-free monic polynomials
$f\in F[X]$. Let $f\in F[X]$ be an invariant factor of $\alpha$ of degree $d$. Since every irreducible polynomial over $F$ has distinct roots in any splitting field and $f$ is square-free, it follows that $f$ has distinct roots in a splitting field, say $K$.
In particular, $C_f$ is similar in $\GL_d(K)$ to a diagonal matrix $\mathrm{diag}(\lambda_1,\dots,\lambda_d)$.
Thus $C_f^p$ is similar in $\GL_d(K)$ to $\mathrm{diag}(\lambda_1^p,\dots,\lambda_d^p)$. Now the map $\Omega:K\to K$
$$
\Omega(k)=k^p,\quad k\in K,
$$
is an automorphism of $K$ with fixed field $F$. We also write $\Omega$ for the associated automorphism $K[X]\to K[X]$.
Since $f\in F[X]$, we see that $f=\Omega(f)=(X-\lambda_1^p)\cdots (X-\lambda_d^p)$ is the minimal
polynomial of $C_f^p$. Thus $C_f^p$ is similar to $C_f$ in $\GL_d(F)$, whence $\alpha$ and $\alpha^p$ have the same rational canonical form.
\end{proof}

Given a field $K$, and $K$-vector spaces $V_1$ and $V_2$ with automorphisms $\a_1$ and $\a_2$, respectively, we say that $\a_1$
and $\a_2$ are similar if there is an isomorphism $f:V_1\to V_2$ such that $f\a_1 f^{-1}=\a_2$.

\begin{lemma}\label{pid} Let $R$ be a principal ideal domain, $M$ a non-zero finitely generated torsion $R$-module, and $q\in R$
an irreducible element. Then the isomorphism type of $M$ is completely determined by the composition length of $M$ and the isomorphism type of $N=qM$.
In particular, if $K$ is a field, $W$ is a non-zero finite dimensional $K$-vector space, $q\in K[X]$ is irreducible, and $u,v\in\mathrm{End}(W)$
are such that $u|_{q(u)W}$ and $v|_{q(v)W}$ are similar, then $u$ and $v$ are similar.
\end{lemma}

\begin{proof} For any $r\in R$, set
$$
M_r=\{x\in M\,|\, r^t x=0\text{ for some }t\geq 1\},\, N_r=\{x\in N\,|\, r^t x=0\text{ for some }t\geq 1\}.
$$
Let $(q_1^{a_1}\cdots q_s^{a_s})$ be the annihilating ideal of $M$, where $q_1,\dots,q_s$ are non-associate irreducible elements of $R$, and each $a_i\geq 1$.
If $q$ is non-associate to every $q_i$, then $M=N$ and there is nothing to do. We assume henceforth that $q$ is associate to some $q_i$, say $q_1$. Then
$M_{q_j}=N_{q_j}$ for every $j>1$. Moreover, we have
$$
M=M_{q_1}\oplus\cdots\oplus M_{q_s}.
$$
It remains to see that the isomorphism type of $M_q=M_{q_1}$ is determined by that of $N_q$ and the composition length of $M$. There exists a unique sequence of non-negative integers $(e_1,e_2,\dots)$,
which is eventually 0, and such that 
$$
M_q=U_1\oplus U_2\oplus\cdots,
$$
where each $U_i$ is the direct sum of $e_i$ cyclic submodules with annihilating ideal $(q^i)$. The corresponding sequence for $N_q$ is clearly $(e_2,e_3,\dots)$. Thus $e_2,e_3,\dots$
are determined by the isomorphism type of $N$. Moreover,
the composition length of $M$ is equal to the sum of the composition lengths of the $M_{q_j}$, $j>1$, plus $e_1+2e_2+3e_3+\cdots$, so $e_1$ is determined by the composition length of $M$ and
the isomorphism type of $N$. This proves the first statement.

As for the second, take $R=K[X]$, and view $M=W$ and $N=q(u)W$ as $R$-modules via $u$. The similarity type of $u$ (resp. $u|_N$) is completely determined by the isomorphism type of $M$ (resp. $N$) as 
$R$-module. Moreover, in the above notation, $e_2,e_3,\dots$
are determined by the isomorphism type of $N$, and
$$
\dim_K M=(\deg q)(e_1+2e_2+3e_3+\cdots)+\dim_K M_{q_2}+\cdots+\dim_K M_{q_s},
$$
so $e_1$ is determined by the fixed quantities $\dim_K M$ and $\deg q$, together with the isomorphism type of $N$ as $R$-module. This proves the second statement.
\end{proof}

\section{Background from group theory}\label{gt}

Given a group $G$ with subgroups $H$ and $K$, we say that $H$ and $K$ are conjugate in $G$ if there is $g\in G$ such that
$gHg^{-1}=K$, in which case we write $H\sim  K$. 

\begin{lemma}\label{conjugados} Let $G$ be a group, and suppose that $H_1,H_2$ are conjugate subgroups of $\Aut(G)$. Then
$\Hol(G,H_1)$ and  $\Hol(G,H_2)$ are conjugate subgroups of $\Hol(G)$, and are therefore isomorphic.
\end{lemma}

\begin{proof} By assumption, there is $\gamma\in\Aut(G)$ such that $\gamma H_1\gamma^{-1}=H_2$. Then $\gamma\in \Hol(G)$
and we have $\gamma G\gamma^{-1}=\gamma(G)=G$  inside of $\Hol(G)$. Therefore
$$
\gamma \Hol(G,H_1)\gamma^{-1}=\gamma (G\rtimes H_1)\gamma^{-1}=\gamma(G)\rtimes \gamma H_1\gamma^{-1}=G\rtimes H_2=\Hol(G,H_2).
\qedhere
$$
\end{proof}

\begin{lemma}\label{ord} Let $G$ be a finite group having a normal subgroup $N$ such that $\gcd(|G/N|,|N|)=1$.
Then any subgroup $K$ of $G$ such that $|K|$ is a factor of $|N|$ must be included in $N$. In particular, if $x\in G$
is such that $o(x)$ is a factor of $|N|$, then $x\in N$.
\end{lemma}

\begin{proof} As $KN/N\cong K/(K\cap N)$, we have $|KN/N|\mid |K|$. But $KN/N$ is a subgroup of $G/N$, so $|KN/N|\mid |G/N|$.
Since $\gcd(|G/N|,|N|)=1$, we infer that $KN/N$ is trivial, so $K\subseteq N$.
\end{proof}

\begin{lemma} \label{nicecase0} Let $A$ be an abelian group and let $H$ and $K$ be subgroups of $\mathrm{Aut}(A)$.
Suppose that $f:\Hol(A,H)\to\Hol(A,K)$ is an isomorphism such that $f(A)=A$. Then $H\sim K$. In particular, if $A$
is finite and abelian, and $\gcd(|A|,|H|)=1$, then $\Hol(A,H)\cong\Hol(A,K)$ forces $H\sim K$.
\end{lemma}

\begin{proof} By hypothesis $f$ restricts to an automorphism $u$ of $A$ and induces an isomorphism
$g:\Hol(A,H)/A\to \Hol(A,K)/A$. On the other hand,
there are isomorphisms $i:H\to \Hol(A,H)/A$ and $j:K\to \Hol(A,K)/A$.
Let $v:H\to K$ be the isomorphism given by $v=j^{-1}gi$.

Given any $h\in H$, we have $f(h)=b k$ for unique $b\in A$ and $k\in K$ and, by definition,
$v(h)=k$. Since $A$ is abelian, conjugation by $k=v(h)$ and $bk=f(h)$ agree on $A$. Thus,
for any $a\in A$, 
\begin{equation}\label{typ2}
u((h)(a))= u(hah^{-1})=f(hah^{-1})=f(h)f(a)f(h)^{-1}=
v(h)u(a)v(h)^{-1}=v(h)(u(a)).
\end{equation}
Therefore $u h u^{-1}=v(h)$ for all $h\in H$. As $v$ is an isomorphism, we infer $u H u^{-1}=K$, which
proves the first part. As for the second, suppose that $A$
is finite and abelian, $\gcd(|A|,|H|)=1$, and that $s:\Hol(A,H)\to\Hol(A,K)$ is an isomorphism. Then
$s(A)$ is a normal subgroup of $\Hol(A,K)$ of order $|A|$, where $|A|$ relatively prime to $|H|=|K|=|\Hol(A,K)/A|$.
It follows from Lemma \ref{ord} that $s(A)=A$, whence $H\sim K$ by the first part.
\end{proof}

The condition that $A$ be abelian is not necessary for the second part of Lemma \ref{nicecase0},
although we will not require this more powerful result.
The condition that  $f(A)=A$ cannot be removed with impunity from the first part of Lemma \ref{nicecase0}, as Examples \ref{final} and 
\ref{e9} show.

\begin{cor}\label{suma} Let $A$ be a finite (additive) abelian group and let $H$ and $K$ be abelian subgroups of $\mathrm{Aut}(A)$.
Suppose that $\Hol(A,H)\cong\Hol(A,K)$ and that the sum of all subgroups $(h-1)(A)$, as $h$ runs through~$H$, is equal to $A$.
Then $H\sim K$.
\end{cor}

\begin{proof} Let $f:\Hol(A,H)\to\Hol(A,K)$ be an isomorphism. 
The stated hypotheses imply that the derived subgroup of $\Hol(A,H)$ (resp. $\Hol(A,K)$) is $A$ (resp.
a subgroup of $A$). Thus $f$ maps $A$ inside of $A$. But $A$ is finite, so $f(A)=A$ and Lemma \ref{nicecase0}
applies.
\end{proof}

Examples \ref{final} and \ref{e9} show that Corollary \ref{suma} fails if the stated condition on $H$ is removed.

\section{Isomorphism of relative holomorphs forces conjugacy of the complements}\label{positive}



We are ready to prove our main result.

\begin{theorem}\label{lindo} If $\alpha,\beta\in\GL(V)$, then $\Hol(V,\alpha)\cong\Hol(V,\beta)$ if and only if
$\langle\alpha\rangle\sim\langle\beta\rangle$.
\end{theorem}

\begin{proof} For $\pi\in\GL(V)$, set $G_\pi=\Hol(V,\pi)$ and let $V_\pi=[G_\pi,G_\pi]$ be the derived subgroup of $G_\pi$.
We readily see that $V_\pi=(\pi-1)(V)$.

If $\langle\alpha\rangle\sim\langle\beta\rangle$ then $G_\alpha\cong G_\beta$ by Lemma \ref{conjugados}.
Suppose next $G_\alpha\cong G_\beta$. 

If $\gcd(p,o(\alpha))=1$ then $\langle\alpha\rangle\sim\langle\beta\rangle$ by Lemma \ref{nicecase0}.
We suppose henceforth that $p\mid o(\alpha)$. 

There is a group isomorphism, say $f:G_\alpha\to G_\beta$, sending $V_\alpha$ onto $V_\beta$. 
We have $f(\a)=w\gamma$ for unique $w\in V$ and $\gamma\in\langle\b\rangle$. It is clear that
the subspaces $V_\a$ and $V_\b$ are invariant under $\a$ and $\gamma$, respectively. A calculation analogous to (\ref{typ2}) shows that
$\alpha|_{V_\alpha}$ is similar to $\gamma|_{V_\beta}$. Indeed, let $u:V_\a\to V_\b$ be the isomorphism of $F$-vector
spaces induced by $f$. Then for
any $v\in V_\alpha$, we have
$$
u(\alpha(v))=f(\alpha(v))=f(\alpha v \alpha^{-1})=f(\alpha)f(v) f(\alpha)^{-1}=\gamma f(v) \gamma^{-1}=\gamma(f(v))=\gamma(u(v)),
$$
so that $u\alpha|_{V_\alpha} u^{-1}=\gamma|_{V_\beta}$, as claimed.


Since $f(V_\alpha)=V_\beta$, we see that $f$ also induces an isomorphism $g:G_\alpha/V_\alpha\to G_\beta/V_\beta$. As $V\cap\langle\alpha\rangle$
is trivial, we have
$$
o(\beta)=o(\alpha)=o(V_\alpha\alpha)=o(g(V_\alpha\alpha))=o(V_\beta w \gamma).
$$
Let $m$ denote this common number. Since $\gamma\in\langle\beta\rangle$, we see that
$o(\gamma)\mid m$. Set
$$
h=1+X+\cdots+X^{p-1}=\frac{X^p-1}{X-1}=\frac{(X-1)^p}{X-1}=(X-1)^{p-1}\in F[X].
$$
Then
\begin{equation}
\label{wp}
(w\gamma)^p=h(\gamma)(w)\gamma^p\text{ and }h(\gamma)(w)\in (\gamma-1)(V)\subseteq V_\beta.
\end{equation}

Suppose first that $p\mid o(\gamma)$. Then (\ref{wp}) implies $(V_\beta w \gamma)^{o(\gamma)}=V_{\beta}$. Thus $m\mid o(\gamma)$
and therefore $m=o(\gamma)$, whence $\langle\gamma\rangle=\langle\beta\rangle$. It follows that
$G_\beta=G_\gamma$ and $V_\beta=V_\gamma$. Thus $\alpha|_{V_\alpha}$ is similar to $\gamma|_{V_\gamma}$, which 
implies that $\alpha$ is similar to $\gamma$ by Lemma \ref{pid}.

Suppose next that $p\nmid o(\gamma)$. Then
$p\nmid o(\gamma|_{V_\beta})$. As $\alpha|_{V_\alpha}$ is similar to $\gamma|_{V_\beta}$, we deduce that 
$p\nmid o(\alpha|_{V_\alpha})$.

The map $f^{-1}:G_\beta \to G_\alpha$ is also an isomorphism, and $f^{-1}(\beta)=z\delta$ for unique $z\in V$ and $\delta\in 
\langle\alpha\rangle$. If $p\mid o(\delta)$ we may deduce, as above, that $\beta$ is similar to $\delta$, where 
$\langle\delta\rangle=\langle\alpha\rangle$. Suppose next that $p\nmid o(\delta)$, which implies, as above, that
$p\nmid o(\beta|_{V_\beta})$.

Now (\ref{wp}) implies $m\mid o(\gamma)p$. Since $p\mid o(\a)$, then $p$ and $o(\gamma)$ are
relatively prime factors of $m$, so $o(\gamma)p|m$ and therefore $m=o(\gamma)p$. Thus $\gamma=\beta^{pj}$, where $j\in\Z$
and $\gcd(pj,m)=p$.
We may thus write $pj=p^s k$, where $s\geq 1$ and $k\in\Z$ is relatively prime to~$m$. Therefore 
$\langle\beta\rangle=\langle\beta^k\rangle$, $V_\beta=V_{\beta^k}$, and $o(\beta|_{V_\beta})=o(\beta^k|_{V_\beta})$.
Since $\gamma=(\beta^k)^{p^s}$ and $p\nmid o(\beta^k|_{V_\beta})$, Lemma \ref{psimilar} implies that $\beta^k|_{V_\beta}$ is similar 
to~$\gamma|_{V_\beta}$. But $\gamma|_{V_\beta}$ is similar to $\alpha|_{V_\alpha}$, so 
$\beta^k|_{V_{\beta^k}}$ is similar to $\alpha|_{V_\alpha}$. Lemma \ref{pid} yields that
$\alpha$ is similar to $\beta^k$.
\end{proof}

\section{The unipotent case}\label{unicase}


For a group $G$, its lower central series $G^1,G^2,\dots$ is inductively defined by $G^1=G$ and 
$G^{i+1}=[G,G^{i}]$, $i\geq 1$. If $G^{i+1}=1$ for some $i\geq 1$, we say that $G$ is nilpotent, and the smallest such $i$ is called
the nilpotency class of $G$. It is well-known that every finite $p$-group is nilpotent. 

Given an additive abelian group $A$ and an endomorphism $\a$ of $A$, we say that $\a$ is unipotent if 
$\alpha=1+\beta$,  with $\beta\in \mathrm{End}(A)$ nilpotent.

\begin{lemma}\label{unip} Let $A$ be a non-trivial finite abelian $p$-group. Let $\alpha\in\mathrm{Aut}(A)$ and set 
$G_\alpha=\Hol(A,\alpha)$. Then
 $G_\alpha$ is nilpotent $\Leftrightarrow \alpha$ is unipotent $\Leftrightarrow$ the order of $\alpha$ is a power of $p$.
\end{lemma}

\begin{proof} It is readily seen that the proper terms of the lower central series of $G_\alpha$ are $(\alpha-1)^i A$, $i\geq 1$,
so  $G_\alpha$ is nilpotent $\Leftrightarrow \alpha$ is unipotent. If the order of $\alpha$ is a power of $p$, then $G_\alpha$
is a $p$-group, and hence nilpotent. Suppose $\alpha$ is unipotent, so $\alpha=1+\beta$,  with $\beta\in \mathrm{End}(A)$
nilpotent, and therefore $\beta^{p^m}=0$ for some $m\geq 0$. We show by induction on $m$ that this implies that the order of $1+\beta$ 
is a power of~$p$. The case $m=0$ is trivial.  Suppose $m>0$ and the result is true for $m-1$. We have $p^\ell A=0$ for some $\ell\geq 1$.
Since $p^\ell\mid \binom{p^\ell}{i}$
for $0<i<p$, we have $(1+\beta)^{p^\ell}=1+\beta^p\gamma$ for some $\gamma\in \Z[\beta]$. As $(\beta^p\gamma)^{p^{m-1}}=0$,
the order of $(1+\beta)^{p^\ell}$ is a power of $p$, and hence so is that of $1+\beta$.
\end{proof}

\begin{theorem}\label{lindo2} Suppose $\alpha,\beta\in\GL(V)$ are unipotent and that $\Hol(V,\alpha)\cong\Hol(V,\beta)$. Then
$\alpha$ is similar to $\beta$.
\end{theorem}

\begin{proof} Set $G_\alpha=\Hol(V,\alpha)$ and $G_\beta=\Hol(V,\beta)$. Then $G_\alpha$ and $G_\beta$ are nilpotent
by Lemma \ref{unip}, and we let $m$ be the common nilpotency class
of $G_\alpha$ and $G_\beta$.

As indicated earlier, the proper terms of the lower central series of $G_\alpha$ (resp. $G_\beta$)
are $(\alpha-1)^i V$ (resp. $(\beta-1)^i V$), $i\geq 1$.

There is basis of $V$ relative to which the matrix of $\alpha$ (resp. $\beta$) is
the direct sum of $e_1,\dots,e_m$ (resp. $f_1,\dots,f_m$) Jordan blocks with eigenvalue 1 of sizes $1,\dots, m$, respectively.  Thus
$$e_1+2e_2+\cdots+m e_m=\dim V=f_1+2f_2+\cdots+m f_m.
$$

This yields a basis of $(\alpha-1)V$ (resp. $(\beta-1)V$) relative to which the matrix of $\alpha$ (resp. $\beta$) is
the direct sum of $e_2,\dots,e_m$ (resp. $f_2,\dots,f_m$) Jordan blocks with eigenvalue 1 of sizes $1,\dots, m-1$. Since
$[G_\alpha,G_\alpha]=(\alpha-1)V$ (resp. $[G_\beta,G_\beta]=(\beta-1)V$), it follows that
$$e_2+2e_3+\cdots+(m-1) e_m=\dim (\alpha-1)V=\dim (\beta-1)V=f_2+2f_3+\cdots+(m-1) f_m.
$$
Continuing in this way we see that the column vector with entries $e_1-f_1,\dots,e_m-f_m$ is annihilated by
the upper triangular matrix with equal entries along each superdiagonal, these entries being $1,2,\dots,m$, 
in each successive superdiagonal.
As this matrix is invertible, we see that $e_1=f_1,\dots, e_m=f_m$, whence $\alpha$ is similar to $\beta$.
\end{proof}



\begin{prop}\label{abe} Suppose $H$ and $K$ are abelian subgroups of $\GL(V)$ such that $\Hol(V,H)\cong\Hol(V,K)$. Then $H\cong K$.
\end{prop}

\begin{proof} Set $V_H=[\Hol(V,H),\Hol(V,H)]$ and $V_K=[\Hol(V,K),\Hol(V,K)]$, which are subgroups of $V$. Let $f:\Hol(V,H)\to \Hol(V,K)$
be an isomorphism. Then $f$ restricts to an isomorphism between $V_H$ and~$V_K$. This yields an isomorphism between
$\Hol(V,H)/V_H$ and $\Hol(V,K)/V_K$. But $$\Hol(V,H)/V_H\cong (V/V_H)\times H,\; \Hol(V,K)/V_K\cong (V/V_K)\times K.$$
Here $V/V_H\cong V/V_K$, as both are $F$-vector spaces of the same dimension, 
so the uniqueness part of the fundamental theorem of finite abelian groups, yields $H\cong K$. 
\end{proof}

\section{Examples}\label{secexa} 


Given a group $G$ and $x,y\in G$, we set ${}^x y=xyx^{-1}$. If $H$ and $N$ are groups such that $H$ acts on~$N$ by automorphisms
via a homomorphism $T:H\to\Aut(N)$, we may consider the semidirect product $N\rtimes_{T} H$, where
$$
{}^h x=hxh^{-1}=T(h)(x),\quad h\in H,x\in N.
$$

There is a slight generalization of Theorem \ref{lindo} to a semidirect product
$V\rtimes_T\langle \a\rangle$, where the action of $\langle \a\rangle$ on $V$ is not necessarily faithful.
We omit the details of the proof, which is similar to that of Theorem \ref{lindo} but notationally more complicated.

We let $U_n(p)$ stand for the subgroup of $\GL_n(p)$ consisting of 
all upper triangular matrices with 1's along the  main diagonal. For later reference 
we observe that $U_3(p)=\mathrm{Heis}(p)$, the Heisenberg group over $F$.

Throughout this section, we set $R=\Z/p^m\Z$, where $m\geq 1$, and let $U$ be a free 
$R$-module of finite rank $n>0$. Given a subgroup $H$ of $\GL_n(R)$, taking $U=R^n$,
we may consider the semidirect product $U\rtimes_T H$, where $T$ is the homomorphism $H\hookrightarrow \GL_n(R)\to\GL(U)$,
and $\GL_n(R)\to\GL(U)$ is the isomorphism associated to the canonical basis of $U$, so that ${}^h v=h\cdot v$ is the product
of the $n\times n$ matrix $h$ by the vector $v$ of length $n$. It is clear that $U\rtimes_T H\cong \Hol(U,T(H))$,
and we will write $\Hol(U,H)=U\rtimes_T H$ from now on.

\begin{lemma}\label{matriz} Let $H$ be a subgroup of $\GL_n(R)$ and set $U=R^n$. Suppose there are subgroups $W$ and~$K$ of $\Hol(U,H)$
such that $W$ is normal, $W\cong U$, $\Hol(U,H)=WK$, and $W\cap K$ is trivial. Assume that the homomorphism $K\to\GL(W)$ arising from 
the conjugation action of $K$
on $W$ is faithful. Fix an $R$-basis $B=\{w_1,\dots,w_n\}$ of $W$, and let $L$ be the image of the corresponding homomorphism 
$v:K\to\GL(W)\to\GL_n(R)$.
Then $\Hol(U,H)\cong\Hol(U,L)$.
\end{lemma}

\begin{proof} For $w\in W$, let $[w]_B\in U$ be the coordinates  of $w$ relative to $B$. Thus, the map $u:W\to U$,
given by $w\mapsto [w]_B$, is an isomorphism. By assumption, the map $K\to L$, given by $k\mapsto v(k)$ is an isomorphism.
Consider the map $f:W\rtimes K\to\Hol(U,L)$ given by $wk\mapsto u(w)v(k)$. We claim that $f$ is an isomorphism.
It suffices to verify that $f$ maps ${}^k w$ into ${}^{v(k)} u(w)$, i.e.,  that
$[{}^k w]_B=v(k)\cdot [w]_B$. To see this, observe that by definition, if $k\in K$
and $1\leq j\leq n$, then $${}^k w_i=\underset{1\leq j\leq n}\sum v(k)_{ji}w_j,$$ so that
$[{}^k w_i]_B$ is the $i$th column of $v(k)$, that is, $[{}^k w_i]_B=v(k)\cdot [w_i]_B$. As $K$ acts linearly on $W$,
it follows that $[{}^k w]_B=v(k)\cdot [w]_B$. 
\end{proof}

\begin{exa}\label{final} Take $p=2$ and $m=3$, so that $R=\Z/8\Z$, and further take $n=4$ and $U=R^4$, whose elements are viewed as column vectors. Let 
$$
A=\begin{pmatrix} 3 & -1 & 1 & -2\\0 & 3 & -3 & 1\\0 & 3 & 4 & 3\\2 & 0 & -2 & 3\end{pmatrix}\in\GL_4(R),
$$
and set $G=\Hol(U,A)=U\rtimes\langle A\rangle$, where
$
AvA^{-1}=A\cdot v,
$
the product of $A$ by the column vector $v\in U$. Then for $S\in \GL_4(R)$, we have $G\cong\Hol(U,S)$
if and and only if $\langle S\rangle\sim\langle A\rangle$ or $\langle S\rangle\sim\langle B\rangle$, where
$\langle B\rangle\not\sim\langle A\rangle$ and
\begin{equation}
\label{defb}
B=\begin{pmatrix} -1 & -2 & 2 & 4\\0 & 3 & -3 & 1\\0 & 3 & 4 & 3\\1 & 0 & -2 & 3\end{pmatrix}\in\GL_4(R).
\end{equation}

Indeed, for $M\in\GL_4(R)$ we denote by $\overline{M}$ the image of $M$ under the canonical
projection $\GL_4(R)\to \GL_4(F)$. The characteristic polynomial of $\overline{A}$
is $f(X)=(X+1)^2(X^2+X+1)$. Thus the minimal polynomial of $\overline{A}$ is $f(X)$ or
$g(X)=(X+1)(X^2+X+1)=X^3+1$. In the latter case $\overline{A}$ has order 3, which is easily seen to be false.
Thus the minimal polynomial of $\overline{A}$ is $f(X)$. As the degree of $f(X)$ is the size of $\overline{A}$, it follows that 
$\overline{A}$ is similar to the companion matrix of $f(X)$, or the direct sum of the companion matrices of
$f_1(X)=X^2+1$ and $f_2(X)=X^2+X+1$. Thus the order of $\overline{A}$ is 6, whence the order of $A$ is a multiple of 6.
Exactly the same comments apply to~$B$. In particular, $\overline{A}$ is similar to $\overline{B}$. But
$\langle A\rangle\not\sim\langle B\rangle$, since the determinant of $A$ is -1 and that of $B$ is 3. 
Thus, every odd power of $B$ has determinant 3,
so no odd power of $B$ is similar to~$A$. Therefore
$\langle A\rangle$ and $\langle B\rangle$ are not conjugate in $\GL_4(R)$. 

We next show that $\Hol(U,A)\cong \Hol(U,B)$. For this purpose, note that
$$
A^2=\begin{pmatrix} -3 & -3 & -2 & -2\\2 & 0 & 1 & -3\\-2 & -3 & 1 & 0\\4 & 0 & 4 & -1\end{pmatrix},
A^3=\begin{pmatrix} 3 & 4 & 2 & -1\\0 & 1 & 4 & -2\\2 & 4 & 3 & 4\\2 & 0 & -2 & 1\end{pmatrix},
A^6=\begin{pmatrix} 3 & 0 & -2 & 4\\4 & 1 & 4 & 4\\4 & 0 & -3 & -2\\4 & 0 & 4 & -1\end{pmatrix},
$$
which confirms that the order of $\overline{A}$ is 6. Moreover,
$$
A^{12}=\begin{pmatrix} 1 & 0 & 0 & 4\\0 & 1 & 0 & 0\\0 & 0 & 1 & 0\\0 & 0 & 0 & 1\end{pmatrix},
$$
so the order of $A$ is 24. Set $G=\Hol(U,A)$,
$$
x=\begin{pmatrix} -1 \\0\\0\\ 1\end{pmatrix}\in U,\; y=(x,A^{12})\in G,
$$
where this notation will be used to avoid confusion when dealing with elements of $G$ which are neither in $U$ nor in $\langle A\rangle$.
We further let $W$ be the subgroup of $G$ generated by $[G,G]$ and $y$. As $W$ contains $[G,G]$, it is a normal
subgroup of $G$. We claim that $W$ is abelian and, in fact, isomorphic to $U$. To see this, note that
$[G,G]$ is the subgroup of $U$ generated by the columns of $A-1$, say $f_1,f_2,f_3,f_4$. Thus $W$ is abelian
if and only if $A^{12}$ commutes with $[G,G]$, which means that $A^{12}$
fixes the columns of $A-1$, that is,
$$
(A^{12}-1)(A-1)=0,
$$
which is true as the fourth row of $A-1$ is annihilated by 4. Row reducing $A-1$ we find
that 
$$
\langle f_1,f_2,f_3,f_4\rangle=\langle f_1\rangle\oplus \langle f_2\rangle\oplus \langle f_3\rangle\oplus \langle f_4\rangle,
$$
where $f_2,f_3,f_4$ have order 8 and $f_1$ has order 4. In particular, we have that $[G,G]\cong C_8^3\times C_4$ and 
$G/[G,G]\cong C_2\times C_{24}$.
Moreover, 
$$
y^2=(x,A^{12})^2=x+A^{12}\cdot x=(1+A^{12})\cdot x=f_1,
$$ 
whence $W\cong U$ and $W/[G,G]\cong C_2$. Note that $W$ is complemented by $K=\langle A\rangle$ in $G$, since their intersection is trivial,
so their product has the right order. Also $x,f_2,f_3,f_4$ are the columns of an invertible matrix, say $Q$, so they generate $U$.

The conjugation action of $\langle A\rangle$
on $W$ is faithful, for if $A^i$ acts trivially on $W$, then $(A^i-1)Q=0$, whence $A^i=1$.

Let $M$ be the matrix of the conjugation action of $A$ on $W$ relative to the $R$-basis $\{y,f_2,f_3,f_4\}$. We claim that $M=B$,
as given in (\ref{defb}). Indeed, denoting by $C_i(P)$ the $i$-column of a matrix $P$,
$$
{}^A f_i=A\cdot f_i=A\cdot C_i(A-1)=C_i(A(A-1))=(A-1) C_i(A)=A_{1i} f_1+A_{2i} f_2+A_{3i} f_3+ A_{4i} f_4,
$$
for all $i\in\{2,3,4\}$. Recalling that $y^2=f_1$, it follows that $M$ and $B$ share the last 3 columns.
Let us verify that $M$ and $B$ share the first column. To see this, observe that
$${}^A y=(A\cdot x, A^{12})=((A-1)\cdot x,1)y.
$$
Here the definition of $x$ gives
$$
(A-1)x= -f_1+f_4,
$$
where $f_1=y^2$, so ${}^A y=y^{-2}f_4 y=y^{-1}f_4$, so the first columns of $M$ and $B$ are identical. It follows from Lemma \ref{matriz} that
$\Hol(U,A)\cong\Hol(U,B)$.

Let $S\in\GL_4(R)$ and suppose that $H=\Hol(U,S)$ is isomorphic to $G$. We proceed to show that
$\langle S\rangle\sim\langle A\rangle$ or $\langle S\rangle\sim\langle B\rangle$. By hypothesis,
we have an isomorphism $\Delta:H\to G$.
Here $U$ is a normal
subgroup of $H$ containing $[H,H]$; $U$ is complemented in $H$ by the cyclic subgroup $\langle S\rangle$ 
of order 24; and relative to the canonical basis of $U$, the matrix of the action of $S$ on $U$ by conjugation is $S$.

Set $N=\Delta(U)$ and $T=\Delta(S)$. Then $N$ is a normal subgroup of $G$ isomorphic to $U$ containing $[G,G]$; 
$N$ is complemented in $G$ by the cyclic subgroup $\langle T\rangle$ of order 24;
and relative to some $R$-basis of~$N$, the matrix of the action of $T$ on $N$ by conjugation is
$S$.

As indicated above, $[G,G]\cong C_8^3\times C_4$, so $N/[G,G]\cong C_2$. As $G/[G,G]\cong C_2\times C_{24}$,
there are exactly 3 normal subgroups of $G$ containing $[G,G]$
as a subgroup of index 2, and $N$ must be one of them. The first possibility is 
$N=U$, in which case $\langle S\rangle\sim\langle A\rangle$ by Lemma \ref{nicecase0}.
The second possibility is
$$
N=[G,G]\times\langle A^{12}\rangle\cong C_8^3\times C_4\times C_2\not\cong U,
$$
which cannot be. It remains to analyze the third possibility, namely $N=W$. Since $G=W\rtimes\langle T\rangle$, the order of $T$
modulo $W$ is also 24. As $G=W\rtimes\langle A\rangle$, it follows that
$T=w A^i$, where $w\in W$ and $i\in\Z$ is relatively prime to 24. Since $M=B$, the conjugation action of $T$ on $W$ relative to the $R$-basis
$\{y,f_2,f_3,f_4\}$ of $W$ is $B^i$. But relative to some $R$-basis of $W$, the matrix of the conjugation action of $T$ on $W$ is $S$.
Thus $S$ is similar to $B^i$ with $\gcd(24,i)=1$, as required.
\end{exa}

\begin{exa}\label{e9}  Suppose that $n\geq 6$ and set $V=F^n$. Then there are abelian subgroups $H$ and $L$ of $U_n(p)$
such that $\Hol(V,H)\cong \Hol(V,L)$ but $H\not\sim L$ in $\GL_n(p)$.

Indeed, suppose first that $n=6$ and let $\{v_1,v_2,v_3,v_4,v_5,v_6\}$ be the canonical basis of $V$.
For $A\in M_3(F)$ set
$$
S_A=\begin{pmatrix} I_3 & A\\0 & I_3\end{pmatrix}\in U_6(p),
$$
and let
$$
A_1=\begin{pmatrix} 1 & 0 & 0 \\0 & 1 & 0\\0 & 0 & 1\end{pmatrix},
A_2=\begin{pmatrix} 0 & 1 & 0 \\0 & 0 & 0\\0 & 0 & 0\end{pmatrix},
A_3=\begin{pmatrix} 0 & 0 & 1 \\0 & 0 & 0\\0 & 0 & 0\end{pmatrix},
$$
$T_H$
the additive subgroup of $M_3(F)$ generated by $A_1,A_2,A_3$, and
$H$ the subgroup of $U_6(p)$ generated by $S_{A_1},S_{A_2},S_{A_3}$.

Let $W$ be the subgroup of $\Hol(V,H)$ generated by 
$v_1,v_2,v_3,S_{A_1},S_{A_2},S_{A_3}$. Then $W\cong V$. Moreover, $W$ is a normal
subgroup of $\Hol(V,H)$. Let $K$ be the subgroup of $\Hol(V,H)$ generated $v_4,v_5,v_6$. We see that $K\cong H$ and  
$\Hol(V,H)=W\rtimes K$. When $K$ acts on $W$ by conjugation, the matrices corresponding to the actions of $v_4,v_5,v_6$ relative
to the basis $\{v_1,v_2,v_3,S_{A_1},S_{A_2},S_{A_3}\}$ are respectively equal to $S_{B_1},S_{B_2},S_{B_3}$,
where $B_i$ is the opposite of the matrix formed by the $i$th columns of $A_1,A_2,A_3$, in that order, for $i\in\{1,2,3\}$.
Thus
$$
B_1=\begin{pmatrix} -1 & 0 & 0 \\0 & 0 & 0\\0 & 0 & 0\end{pmatrix},
B_2=\begin{pmatrix} 0 & -1 & 0 \\-1 & 0 & 0\\0 & 0 & 0\end{pmatrix},
B_3=\begin{pmatrix} 0 & 0 & -1 \\0 & 0 & 0\\-1 & 0 & 0\end{pmatrix},
$$
so the action of $K$ on $W$ is faithful. Let $L$ be the subgroup of $U_6(p)$ generated by $S_{B_1},S_{B_2},S_{B_3}$, so
that $\Hol(V,H)\cong\Hol(V,L)$ by Lemma \ref{matriz}. Let $T_L$
the additive subgroup of $M_3(p)$ generated by $B_1,B_2,B_3$.

Suppose, if possible, that $H\sim L$ in $\GL_6(p)$. Then, there is $X\in\GL_6(p)$ such that $X H X^{-1}=L$. 
Thus $X$ gives rise to the isomorphism
$f:\Hol(V,H)\to\Hol(V,L)$ given by $vh\mapsto (X\cdot v)(X h X^{-1})$. Then $f$ must map the center of $\Hol(V,H)$
onto the center of $\Hol(V,L)$. But both centers are equal to $\langle v_1,v_2,v_3\rangle$,
so
$$
X=\begin{pmatrix} Y & Q\\0 & Z\end{pmatrix},
$$
where $Y,Z\in\GL_3(p)$ and $Q\in M_3(F)$. Then $X H X^{-1}=L$ gives $Y T_H Z^{-1}=T_L$. But all matrices in $T_L$ have rank at most 2, whereas $T_H$ has a matrix of rank 3.
We deduce that $H\not\sim L$ in $\GL_6(p)$.

In general, take $m=\lceil n/2\rceil$, for $A\in M_m(F)$ set
$$
S_A=\begin{pmatrix} I_m & A\\0 & I_{n-m}\end{pmatrix}\in U_n(p),
$$
and let $A_1,\dots, A_{n-m}\in M_{m,n-m}(p)$ be defined as follows: $A_1=\mathrm{diag}(1,\dots,1)$, where the number of 1's is $n-m$,
$A_2=E^{1,2},\dots, A_{n-m}=E^{1,n-m}$, where 
$E^{i,j}$ is the matrix having a 1 in position $(i,j)$
and 0's elsewhere. We can then continue as above, noting that 
all matrices in $T_L$ have rank at most 2, whereas $T_H$ has a matrix of rank $n-m\geq 3$.

\end{exa}

We will refer to a group $G$ as admitting if there exist non-conjugate subgroups $H$ and $K$ of $\mathrm{Aut}(G)$ such that
$\Hol(N,H)\cong\Hol(N,K)$. We will say that $G$ as highly admitting if there exist non-isomorphic subgroups $H$ and $K$ of $\mathrm{Aut}(G)$ such that $\Hol(N,H)\cong\Hol(N,K)$.

\begin{exa}\label{e3} Suppose that $p$ is odd and $n\geq 2$. Then $V$ is highly admitting.

Indeed, suppose first $n=2$ and set $V=F^2$. Let $\{v_1,v_2\}$ be the canonical basis of $V$
and let $H$ be the subgroup of $\GL_2(p)$ generated by
 $$
A=\begin{pmatrix} -1 & 0\\ 0 & 1\end{pmatrix},
B=\begin{pmatrix} 1 & 1\\ 0 & 1\end{pmatrix}.
$$
Since $o(B)=p$, $o(A)=2$, and ${}^A B=B^{-1}$, we see  that $H$ is the dihedral group of order $2p$.

Let $W$ be the subgroup of $\Hol(V,H)$ generated by $v_1,B$. Then $W\cong V$. Moreover, $W$ is a normal
subgroup of $\Hol(V,H)$. Let $K$ be the subgroup of $\Hol(V,H)$ generated by $v_2,A$. Clearly $K\cong C_{2p}$
is not isomorphic to $H$, and  $\Hol(V,H)=W\rtimes K$. When $K$ acts on $W$ by conjugation, the matrices corresponding to the actions of 
$v_2$ and $A$ relative
to the basis $\{v_1,B\}$ are respectively equal to
$$
\begin{pmatrix} 1 & -1\\ 0 & 1\end{pmatrix},
\begin{pmatrix} -1 & 0\\ 0 & -1\end{pmatrix}.
$$
The subgroup of $\GL_2(p)$ generated by these matrices, say $L$, is isomorphic to $C_{2p}$, so the action of $K$ on $W$
is faithful. It follows from Lemma \ref{matriz} that $\Hol(V,L)\cong\Hol(V,H)$, even though $L$ is not isomorphic to $H$.

The general case when $\{v_1,\dots,v_n\}$ is the canonical basis of $V$ follows by extending $A,B$ so that they fix $v_3,\dots,v_n$.

Besides proving that $V$ is highly admitting when $n\geq 2$ and $p$ is odd, this example shows that even though 
$\Hol(V,H)\cong\Hol(V,L)$ and the Sylow $p$-subgroup of $H$ is conjugate to the Sylow $p$-subgroup of~$L$,
this conjugation cannot be extended to all of $H$ and $L$.
\end{exa}



We next provide an analogue of Example \ref{e3} when $p=2$, provided $n\geq 4$.

\begin{exa}\label{e7} Suppose that $n\geq 4$ and $p=2$. Then $V$ is highly admitting.

Indeed, suppose first $n=4$ and set $V=F^4$. Let $\{v_1,v_2,v_3,v_4\}$ be the canonical basis of $V$ and
let $H$ be the subgroup of $\GL(V)$ generated by
$$
X=\begin{pmatrix} 0 & 0 & 0 & 1\\0 & 0 & 1 & 0\\0 & 1 & 0 & 0\\1 & 0 & 0 & 0\end{pmatrix},
Y=\begin{pmatrix} 0 & 0 & 0 & 1\\0 & 0 & 1 & 1\\1 & 1 & 0 & 0\\1 & 0 & 0 & 0\end{pmatrix},
Z=\begin{pmatrix} 0 & 0 & 0 & 1\\1 & 0 & 1 & 1\\1 & 1 & 0 & 1\\1 & 0 & 0 & 0\end{pmatrix}.
$$
Then $H\cong C_2^3$.
Let $W$ be the subgroup of $\Hol(V,H)$ generated by $v_1+v_4,v_2+v_3,X,Z$. Then $W\cong V$. Moreover, $W$ is a normal
subgroup of $\Hol(V,H)$. Let $K$ be the subgroup of $\Hol(V,H)$ generated by $v_1,Y$. We see that $K\cong D_{8}$ and  
$\Hol(V,H)=W\rtimes K$. When $K$ acts on $W$ by conjugation, the matrices corresponding to the actions of $v_1,Y$ relative
to the basis $\{v_1+v_4,v_2+v_3,X,Z\}$ are respectively equal to
$$
\begin{pmatrix} 1 & 0 & 1 & 1\\0 & 1 & 0 & 1\\0 & 0 & 1 & 0\\0 & 0 & 0 & 1\end{pmatrix},
\begin{pmatrix} 1 & 0 & 0 & 0\\1 & 1 & 0 & 0\\0 & 0 & 1 & 0\\0 & 0 & 0 & 1\end{pmatrix},
$$
The subgroup of $\GL_4(2)$ generated by these matrices, say $L$, is isomorphic to $D_{8}$, so the action of $K$ on $W$
is faithful. It follows from Lemma \ref{matriz} that $\Hol(V,L)\cong\Hol(V,H)$ even though $L$ is not isomorphic to $H$.

The general case when $\{v_1,\dots,v_n\}$ is the canonical basis of $V$ follows by extending $X,Y,Z$ so that they fix $v_5,\dots,v_n$.
\end{exa}

\begin{exa}\label{e1} Conditions (C1)-(C3) below ensure that a group $G$ admits automorphisms $\alpha$ and~$\beta$ such 
that $\Hol(G,\alpha)\cong \Hol(G,\beta)$ but $\langle \alpha\rangle\not\sim\langle \beta\rangle$.

Let $G$ be a group having elements $x$ and $y$ such that: (C1) $o(x)=o(y)$;
(C2) $o(x)=o(x Z(G))$ and $o(y)=o(y Z(G))$ (this means that $\langle x\rangle\cap Z(G)=1=\langle y\rangle\cap Z(G)$);
(C3) there is no automorphism of $G$ that sends $x^i$ to $yz$ for any $i$ relatively prime to the order of $x$ and any $z\in Z(G)$
(this means that when $\Aut(G)$ acts on $\Inn(G)$ by conjugation, the subgroups generated by $i(x)$ and $i(y)$
are in different orbits, where $i(g)\in \Inn(G)$ is the inner automorphism associated to $g\in G$).

Let $\alpha=i(x)$ and $\beta=i(y)$. By (C2), $\Hol(G,\alpha)=G\times\langle u\rangle$, where $u=x^{-1}\alpha$ has the same
order as $x$, and $\Hol(G,\beta)=G\times\langle v\rangle$, where $v=y^{-1}\beta$ has the same
order as $y$. Thus $\Hol(G,\alpha)\cong\Hol(G,\beta)$ by (C1). Also,
$\langle \alpha\rangle$ and $\langle \beta \rangle$ are not conjugate subgroups of $\mathrm{Aut}(G)$ by~(C3).

Many groups satisfy (C1)-(C3). A centerless group $G$ having cyclic subgroups of the same order that
are not in the same $\Aut(G)$-orbit meets (C1)-(C3). For instance: $S_n$, $n\geq 4$, taking $x=(1,2)$ and $y=(1,2)(3,4)$,
and looking at the size of their centralizers; a free group on $n\geq 2$ generators $x_1,\dots,x_n$,
taking $x=x_1$ and $y=[x_1,x_2]$. The general linear group $\GL_n(p)$, with $n\geq 4$ and $p$ odd, is not centerless and
satisfies (C1)-(C3), taking $x=\mathrm{diag}(-1,1,\dots,1)$ and $y=\mathrm{diag}(-1,-1,1,\dots,1)$,
and looking at the size of their centralizers.

As another example, assume $m\geq 2$ and let $G=\Hol(V,\alpha)$, where $\alpha$ acts on $V$, with respect to some basis,
via the direct sum of $m$ copies of the matrix
$$
\begin{pmatrix} 1 & 1\\ 0 & 1\end{pmatrix},
$$
where $n=2m$. Thus the order of $G$ is $p^{2m+1}$. As an abstract group,
\begin{equation}
\label{pres}
G=\langle x_1,x_2,\dots,x_{2m-1},x_{2m},y\,|\, [x_i,x_j]=1, x_i^p=1=y^p, {}^y x_{2i-1}=x_{2i-1}, {}^y x_{2i}=x_{2i} x_{2i-1}\rangle.
\end{equation}
The simplest example occurs when $m=2$ and $p=2$, in which case $|G|=32$. In this case, $G$ can also be described as
a Sylow 2-subgroup $P$ of $\GL_2(\Z/4\Z)$. Indeed, let $N$ be the kernel of the canonical map $\GL_2(\Z/4\Z)\to \GL_2(\Z/2\Z)$,
which consists of all 16 matrices of the form $1+X$, where $X\in M_2(2\Z/4\Z)$. Then $N\cong C_2^4$, and $N$ is  generated by
$$
B=\begin{pmatrix} 1 & 2\\ 0 & 1\end{pmatrix}, C=\begin{pmatrix} 1 & 2\\ 2 & 1\end{pmatrix},
D=\begin{pmatrix} -1 & 0\\ 0 & 1\end{pmatrix}, E=\begin{pmatrix} -1 & 0\\ 0 & -1\end{pmatrix}.
$$
Defining $A\in P$ by
$$
A=\begin{pmatrix} 0 & 1\\ 1 & 0\end{pmatrix},
$$
we see that
$$
{}^A B=BC, {}^A C =C, {}^A D=DE, {}^A E =E.
$$

Using the notation from (\ref{pres}), 
the center of $G$ is generated by all $x_{2i-1}$. Take $x=x_2$ and $y$ as given. Then $x$ and $y$ have order $p$,
which remains the same modulo $Z(G)$. The centralizer of $x^i$, when $p\nmid i$, is the subgroup~$T$ generated by all $x_j$.
The centralizer of $yz$, when $z\in Z(G)$, is equal to $Z(G)\langle y\rangle$. Here $|T|=p^{2m}$ and $|Z(G)\langle y\rangle|=p^{m+1}$.
Since $m\geq 2$, it follows that $T$ cannot be mapped into $Z(G)\langle y\rangle$
by any automorphism of $G$, so $x$ and $y$ meet the required conditions. Note that if $m=1$, then
$G=\mathrm{Heis}(p)$, in which case $G$ does not satisfy (C1)-(C3). Indeed, if $p$ is odd, then any two non-central
elements of $G$ produce subgroups of inner automorphisms of order $p$ that are in the same $\mathrm{Aut}(G)$-orbit; if $p=2$
then $G=D_8$, and $D_8$ is non-admitting (see
Example \ref{e6} below). 

\end{exa}

Our last examples discusses instances of non-admitting finite groups.

\begin{exa}\label{e6}  (a) Suppose that $n=2$ and $p=2$. Then $V$ is non-admitting. 

Indeed, every proper subgroup of $\GL(V)\cong S_3$ is cyclic, and all cyclic subgroups of $S_3$ of the same order are conjugate in $S_3$.

(b) Suppose that $n=3$ and $p=2$. Then $V$ is non-admitting.

Indeed, set $V=F^3$ and let $\{v_1,v_2,v_3\}$ be the canonical basis of $V$.
It is known that the only cases when $\GL_3(2)$ has subgroups of the same order that are not conjugate occur
for orders 4, 12, and 24, and that there are 3 conjugacy classes of groups of order 4 and 2 conjugacy classes of groups of orders 12 and 24. In order 4,
let $H$ and $K$ respectively consist of all matrices of the form
$$
\begin{pmatrix} 1 & 0 & *\\ 0 & 1 & *\\ 0 & 0 & 1\end{pmatrix},
\begin{pmatrix} 1 & * & *\\ 0 & 1 & 0\\ 0 & 0 & 1\end{pmatrix}.
$$
Note that $G_1=\Hol(V,H)$ is not isomorphic to $G_2=\Hol(V,K)$, since $[G_1,G_1]=\langle v_1,v_2\rangle$
and $[G_2,G_2]=\langle v_1\rangle$ (this gives a quick way to verify that $H\not\sim K$). Next let $J$ be the subgroup
generated by the upper triangular Jordan block with eigenvalue 1. Thus $J$ is cyclic of order 4. Set  $G_3=\Hol(V,J)$.
Then $[G_3,G_3]=\langle v_1,v_2\rangle$, but $G_3$ is nilpotent of class 3 and $G_1$ is nilpotent of class 2.
Thus $G_3$ is not isomorphic to $G_1$ or to $G_2$.

In order 24, let $H$ and $K$ respectively consist of all matrices of the form
$$
\begin{pmatrix} A & u\\ 0 & 1\end{pmatrix},
\begin{pmatrix} 1 & w\\ 0 & A\end{pmatrix},
$$
where $A\in\GL_2(2)$, $u$ is a column vector of length 2, and $w$ is a row vector of length 2. 
Note that $G_1=\Hol(V,H)$ is not isomorphic to $G_2=\Hol(V,K)$, since $[G_1,G_1]=\langle v_1,v_2\rangle\rtimes A_4$
and $[G_2,G_2]=V\rtimes A_4$ (this gives a quick way to verify that $H\not\sim K$).

In order 12 the situation is as above, but with $A\in\langle C\rangle$, where $C$ is the companion matrix of the polynomial 
$t^2+t+1\in F[t]$.
The outcome is the same.

(c) In addition to the group $C_2^3$ discussed above, every other group of order 8 is non-admitting.
These cases are easily verified and we omit the details.

(d) In addition to the group $C_8$ mentioned above, every cyclic group $C_{p^n}$ is non-admitting.
The case when $p$ is odd is trivial, and the case when $p=2$ requires routine calculations that we omit.
\end{exa}

\medskip

\noindent{\bf Acknowledgment.} We thank Eamonn O'Brien for Magma calculations, and Allen Herman for 
useful discussions.



\begin{thebibliography}{RBMW}

\bibitem[Be]{Be} I. H. Bekker, 
{\em The holomorphs of torsion-free abelian groups},
Izv. Vys\v{s}. U\v{c}ebn. Zaved. Matematika, 3 (1974) 3--11.


\bibitem[B]{B} W. Burnside, \emph{Theory of groups of finite order}, Cambridge, 1897.

\bibitem[C]{C} J. R. Clay, \emph{Completeness of relative holomorphs of abelian groups},
Rocky Mountain J. Math., 10 (1980) 731--741.


\bibitem[CD]{CD} A. Caranti and F. Dalla Volta, \emph{Groups that have the same holomorph as a finite perfect group},
J. Algebra, 507 (2018) 81--102.


\bibitem[D]{D} R. S. Dark, \emph{A complete group of odd order}, 
Math. Proc. Camb. Philos. Soc., 77 (1975) 21--28.

\bibitem[H]{H} N. C. Hsu, \emph{The holomorphs of free abelian groups of finite rank},
Amer. Math. Monthly, 72 (1965) 754--756.

\bibitem[K]{K} N.F. Kuzennyi, \emph{Isomorphism of semidirect products},
Ukr. Math. J., 26 (1974) 543--547.

\bibitem[Ko]{Ko} T. Kohl, \emph{Groups of order $4p$, twisted wreath products and Hopf-Galois
theory}, J. Algebra, 314 (2007) 72–-74.

\bibitem[Ko2]{Ko2} T. Kohl, \emph{Multiple holomorphs of dihedral and quaternionic groups}, 
Comm. Algebra, 43 (2015) 4290–-4304.


\bibitem[M]{M} G. A. Miller, \emph{On the multiple holomorphs of a group}, Math. Ann., 66 (1908) 133--142.


\bibitem[Mi]{Mi} W. H. Mills, \emph{Multiple holomorphs of finitely generated abelian groups}, Trans. Amer.
Math. Soc., 71 (1951) 379--392.

\bibitem[Mi2]{Mi2} W. H. Mills, \emph{On the non-isomorphism of certain holomorphs}, 
Trans. Am. Math. Soc., 74 (1953) 428--443.


\bibitem[P]{P} W. Peremans, \emph{Completeness of holomorphs}, Nederl. Akad. Wet., Proc., Ser. A, 60 (1957) 608--619.













\bibitem[T]{T} D.R. Taunt, \emph{Remarks on the isomorphism problem in theories of construction of finite groups},
Proc. Cambridge Phil. Soc., 51 (1955) 16--24.


\bibitem[VO]{VO} A. Vera L\'{o}pez and L. Ortiz de Elguea, 
\emph{The conjugacy-vectors of all relative holomorphs of an elementary Abelian group of order 16},
Port. Math.,  47 (1990) 243--257.








\end{thebibliography}
\end{document}